\documentclass[12pt]{amsart}
\usepackage[utf8]{inputenc}
\usepackage[T1]{fontenc}
\usepackage[all]{xy}
\usepackage{lipsum}
\usepackage{url}
\usepackage{tikz}
\usepackage{stackrel}
\usepackage{color}

\usepackage{amsmath,amsthm,amssymb}

\usepackage{xcolor}
\usepackage{graphicx}
\newcommand{\aspas}[1]{``{#1}''}

\newtheorem{theorem}{Theorem}[section]
\newtheorem{lemma}[theorem]{Lemma}
\newtheorem{example}[theorem]{Example}
\newtheorem{proposition}[theorem]{Proposition}
\theoremstyle{definition}
\newtheorem{definition}[theorem]{Definition}

\newtheorem{remark}[theorem]{Remark}
\newtheorem{corollary}[theorem]{Corollary}

\numberwithin{equation}{section}

\begin{document}


\renewcommand{\bf}{\bfseries}
\renewcommand{\sc}{\scshape}

\title[Sectional Category and the parametrized Borsuk-Ulam property]%
{Sectional Category and the parametrized Borsuk-Ulam property}



\author{Cesar A. Ipanaque Zapata}
\author{Daciberg L. Gonçalves} 
\address{Departamento de Matem\'atica, IME Universidade de S\~ao Paulo\\
Rua do Mat\~ao 1010 CEP: 05508-090 S\~ao Paulo-SP, Brazil}
\email{cesarzapata@usp.br}
\email{dlgoncal@ime.usp.br}


\thanks {The first author would like to thank grant\#2022/03270-8 and grant\#2022/16695-7, S\~{a}o Paulo Research Foundation (FAPESP) for financial support. The second  author was partially supported by the FAPESP \aspas{Projeto
Tem\'atico-FAPESP Topologia Alg\'ebrica, Geom\'etrica e Diferencial} 2016/24707-4 (São Paulo-Brazil).}

\begin{abstract} In this paper, for fibrations $p:E\to B$, $p':E'\to B$ over $B$ and a free involution $\tau:E\to E$ which satisfy the equality $p\circ\tau=p$, we discover a connection between the sectional category of the double covers $q:E\to E/\tau$ and $q^Y:F_{p'}(E',2)\to F_{p'}(E',2)/\mathbb{Z}_2$ from the $2$-ordered fibre-wise configuration space $F_{p'}(E',2)$ to its unordered quotient $ F_{p'}(E',2)/\mathbb{Z}_2$, and the parametrized Borsuk-Ulam property (PBUP) for the triple $\left((p,\tau);p'\right)$. Explicitly, we demonstrate that the triple $\left((p,\tau);p'\right)$ satisfies the PBUP if the sectional category of $q$ is bigger than the sectional category of $q^{p'}$. This property connects a standard problem in parametrized Borsuk-Ulam theory to current research trends in sectional category. As applictions of our results,   we explore the PBUP for  $E,  E'$   one of the following fibrations: trivial fibration,  Hopf fibration and the Fadell-Neuwirth fibration. 
\end{abstract}
\maketitle


\section{Introduction}\label{secintro}
In this article \aspas{space} means a topological space, and by a \aspas{map} we will always mean a continuous map. Fibrations are taken in the Hurewicz sense.

Let $\left((p,\tau);p'\right)$ be a triple where $p:E\to B$, $p':E'\to B$ are fibrations over $B$ and $\tau:E\to E$ is a free involution with $p\circ\tau=p$. 
\begin{eqnarray*}
\xymatrix{ 
       E  \ar[rr]^{\tau}\ar[rd]_{p} & & E\ar[ld]^{p} \\
      &B & 
       }
\end{eqnarray*}

Motivated by \cite{dold1988} (cf. \cite{goncalves2023}), we say that $\left((p,\tau);p'\right)$ satisfies the \textit{parametrized Borsuk-Ulam property} (which we shall routinely abbreviate to PBUP) if  for every map $f:E\to E'$ over $B$ (i.e. $f$ satisfies $p'\circ f=p$), 

\begin{eqnarray*}
\xymatrix{ 
       E  \ar[rr]^{f}\ar[rd]_{p} & & E'\ar[ld]^{p'} \\  %
      &B & 
       }
\end{eqnarray*}

\noindent there exists a point $x\in E$ such that $f(\tau(x))=f(x)$.

\begin{remark}
    \noindent\begin{enumerate}
        \item In the case that $p$ is the sphere-bundle of a vector bundle and $p'$ is a vector bundle, both over a paracompact base space $B$, Dold studies the dimension of the set $Z_f=\{e\in E:~f(e)=f(\tau(e))\}$ \cite[Corollary 1.5]{dold1988}.
   \item If $B$ is a point, then $\left((p,\tau);p'\right)$  satisfies the PBUP if and only if  $\left((E,\tau);E'\right)$ satisfies the Borsuk-Ulam property (BUP) as defined in \cite{zapata2023}.         
    \end{enumerate}
\end{remark}

For Hausdorff spaces $X, Y$ and a fixed-point free involution $\tau:X\to X$, in \cite{zapata2023}, the authors discover a connection between the sectional category of the double covers $q:X\to X/\tau$ and $q^Y:F(Y,2)\to D(Y,2)$ from the ordered configuration space $F(Y,2)$ to its unordered quotient $D(Y,2)=F(Y,2)/\Sigma_2$ \cite{fadell1962configuration}, and the Borsuk-Ulam property (BUP) for the triple $\left((X,\tau);Y\right)$. In this paper, we present a connection between sectional category and the parametrized Borsuk-Ulam property.

\medskip The main result of this work  is:

\medskip\noindent \textbf{Theorem 3.10} (Principal theorem). Suppose that $E$ and $E'$ are Hausdorff spaces. If the triple $\left((p,\tau);p'\right)$ does not satisfy the PBUP or the triple $\left((E,\tau);F'\right)$ does not satisfy the BUP then \[\mathrm{secat}(q_E)\leq \mathrm{secat}(q^{p'}).\] Equivalently, if $\mathrm{secat}(q_E)>\mathrm{secat}(q^{p'})$ then the triple $\left((p,\tau);p'\right)$ satisfies the PBUP and the triple $\left((E,\tau);F'\right)$ satisfies the BUP.

\medskip This gives conditions in terms of sectional category to have the PBUP. In the case that $p'=p$ we obtain Remark~\ref{rem:p-p}. Then we give some applications where  the fibrations involved    are one of the following fibrations: trivial fibration,  Hopf fibration and the Fadell-Neuwirth fibration. 

\medskip In  more details the content of this paper is organized as follows: In Section~\ref{sn}, we recall the notion of sectional category (Definition~\ref{defn:sectional-category}) together with basic results about this numerical invariant (Lemma~\ref{prop-sectional-category}, Lemma~\ref{prop-secat-map}, Example~\ref{projective}). A key statement for our objectives is Remark~\ref{pullback}. We present the notion of comparable between involutions (Definition~\ref{comparable}) and we shows that secat is invariant under comparable relation (Lemma~\ref{comparable-invariant}). In Section~\ref{bu}, we present the notion of parametrized Borsuk-Ulam property (Definition~\ref{defn:pbup}) together with basic results about this property via sectional category. Lemma~\ref{bup-secat} is a motivation for the present work. It presents a connection between sectional category and the Borsuk-Ulam property. Remark~\ref{basic-remark} says that if the BUP holds at the level of fibers then the PBUP holds. As a direct consequence, we obtain Proposition~\ref{borsuk-and-parametrized}. Proposition~\ref{inequality-secat-pullback} stays natural inequalities for secat of certain quotient maps. Proposition~\ref{prop:lower-bound-secat-zf} gives a lower bound for $\mathrm{secat}\left(q_{Z_f}\right)$. In Proposition~\ref{top-pbup} we present a topological criterion for the parametrized BUP. It is essential for our purposes. Then 
we have our  main result  which   is Theorem~\ref{theorem-principal}. 
In addition, as a direct consequence of our principal theorem, we have Proposition~\ref{prop:secat-quotient-conf}. In Section~\ref{ap}, we present applications of our results. Remark~\ref{trivial-fibration} shows that Theorem~\ref{theorem-principal} is a generalization, in the parametrized setting, of Lemma~\ref{bup-secat}. In the case that $F\hookrightarrow E\stackrel{p}{\to} B$ is a locally trivial bundle over $B$, Proposition~\ref{prop:upper-bound-secat-fibered} presents an upper bound to $\mathrm{secat}(q^{p})$. Example~\ref{complexHopf-trivial} studies the PBUP between complex Hopf fibration and trivial fibration. To explore the PBUP into the Fadell-Neuwirth fibration we firstly present Remark~\ref{Fadell-Neuwirth-fibration}. Then, in Proposition~\ref{to-fadell}, we explore the PBUP between any fibration and the Fadell-Neuwirth fibration. For instance, in the case of fibrations over the 3-sphere we obtain Example~\ref{s3}. Furthermore, Example~\ref{s2} presents the PBUP between Hopf fibration and the Fadell-Neuwirth fibration. 

\section{Sectional category revisited}\label{sn}
In this section we begin by recalling the notion of sectional category together with basic results about this numerical invariant. We shall follow the terminology in \cite{zapata2023}. If $f$ is homotopic to $g$ we shall denote by $f\simeq g$. The map $1_Z:Z\to Z$ denotes the identity map. 

\medskip Let $p:E\to B$ be a fibration.  A \textit{global cross-section} or \textit{global section} of $p$ is a right inverse of $p$, i.e., a map $s:B\to E$, such that $p\circ s = 1_B$ . Moreover, given a subspace $A\subset B$, a \textit{local section} of $p$ over $A$ is a map $s:A\to E$, such that $p\circ s$ is the inclusion $A\hookrightarrow B$.

\medskip We recall the following definition. 
\begin{definition}\label{defn:sectional-category}
   The \textit{sectional category} of $p$, called originally by Schwarz genus of $p$, see \cite[Definition 5, section 1 chapter III]{schwarz1966}, and denoted by $\mathrm{secat}\hspace{.1mm}(p),$ is the minimal cardinality of open covers of $B$, such that each element of the cover admits a local section to $p$. We set $\mathrm{secat}\hspace{.1mm}(p)=\infty$ if no such finite cover exists.
\end{definition}

This notion was first studied extensively by Schwarz for fibrations in \cite{schwarz1966}. Later by Berstein and Ganea for arbitrary maps \cite[Definition 2.1]{berstein1961}. For our purposes we use sectional category for fibrations as in Definition~\ref{defn:sectional-category}.

\medskip For a commutative ring $R$ and a proper ideal $S\subset R$, the \textit{nilpotency index} of $S$ is given by \[\text{nil}(S)=\min\{k:~\text{ any product of $k$ elements of $S$ is trivial}\}.\] Note that, $\text{nil}(S)$ coincides with $n+1$, where $n$ is the maximum number of factors in a nonzero product of elements from $S$.

\medskip The following statement gives a lower bound in terms of any multiplicative cohomology (see \cite[Proposiç\~{a}o 4.3.17-(3), p. 138]{zapata2022}).

\begin{lemma}\label{prop-sectional-category}
Let $h^\ast$ be a multiplicative cohomology theory and $p:E\to B$ be a fibration, then \[\mathrm{secat}\hspace{.1mm}(p)\geq \mathrm{nil}(\mathrm{Ker}(p^\ast)),\] where $p^\ast:h^\ast(B)\to h^\ast(E)$ is the induced homomorphism in cohomology.
\end{lemma}

\begin{remark}
Lemma~\ref{prop-sectional-category} implies that if there exist cohomology classes $\alpha_1,\ldots,\alpha_k\in h^\ast(B)$ with $p^\ast(\alpha_1)=\cdots=p^\ast(\alpha_k)=0$ and $\alpha_1\cup\cdots\cup \alpha_k\neq 0$, then $\mathrm{secat}\hspace{.1mm}(p)\geq k+1$. 
\end{remark} 

\medskip\noindent Now, note that, if the following diagram

\begin{eqnarray*}
\xymatrix{ E^\prime \ar[rr]^{\,\,} \ar[dr]_{p^\prime} & & E \ar[dl]^{p}  \\
        &  B & }
\end{eqnarray*}
commutes up to homotopy, then $\mathrm{secat}\hspace{.1mm}(p^\prime)\geq \mathrm{secat}\hspace{.1mm}(p)$ (see \cite[Proposition 6, p. 70]{schwarz1966}). Also, from \cite[Proposition 7, g. 71]{schwarz1966}, for any fibration $p:E\to B$ and any map $f:B^\prime\to B$, note that any local section $s:U\to E$ of $p:E\to B$ induces a local section of the canonical pullback $f^\ast p:B^\prime\times_B E\to B^\prime$, called \textit{the local pullback section} $f^\ast(s):f^{-1}(U)\to B^\prime\times_B E$, simply by defining \[f^\ast(s)(b^\prime)=\left(b^\prime,\left(s\circ f\right)(b^\prime)\right).\]

\begin{eqnarray*}
\xymatrix{ &B^\prime\times_B E \ar[rr]^{ } \ar[d]_{f^\ast p} & & E \ar[d]^{p} & \\
       &B^\prime  \ar[rr]_{f} & &  B & \\
       f^{-1}(U)\ar@{^{(}->}[ru]_{}\ar[rr]_{f}\ar@{-->}@/^10pt/[ruu]^{f^\ast(s)} & &U\ar@{^{(}->}[ru]_{}\ar@{-->}@/^10pt/[ruu]^{s} & & }
\end{eqnarray*} Thus, \begin{align}\label{ineq-canonical}
    \text{secat}(f^\ast p)&\leq \text{secat}(p).
\end{align} 

From \cite[Definition 3.4, p. 41]{zapata2023} or \cite[p. 1619]{zapata2023higher} we recall the notion of quasi pullback. 

\begin{definition}
 A \textit{quasi pullback} means a strictly commutative diagram
\begin{eqnarray*}
\xymatrix{ \rule{3mm}{0mm}& X^\prime \ar[r]^{\varphi'} \ar[d]_{f^\prime} & X \ar[d]^{f} & \\ &
       Y^\prime  \ar[r]_{\,\,\varphi} &  Y &}
\end{eqnarray*} 
such that, for any strictly commutative diagram as the one on the left hand-side of~(\ref{diagramadoble}), there exists a (not necessarily unique) map $h:Z\to X^\prime$ that renders a strictly commutative diagram as the one on the right hand-side of~(\ref{diagramadoble}). 
\begin{eqnarray}\label{diagramadoble}
\xymatrix{
Z \ar@/_10pt/[dr]_{\alpha} \ar@/^30pt/[rr]^{\beta} & & X \ar[d]^{f}  & & &
Z\rule{-1mm}{0mm} \ar@/_10pt/[dr]_{\alpha} \ar@/^30pt/[rr]^{\beta}\ar[r]^{h} & 
X^\prime \ar[r]^{\varphi'} \ar[d]_{f^\prime} & X \\
& Y^\prime  \ar[r]_{\,\,\varphi} &  Y & & & & Y^\prime &  \rule{3mm}{0mm}}
\end{eqnarray}   
\end{definition}

Note that such a condition amounts to saying that $X'$ contains the canonical pullback $Y'\times_Y X$ determined by $f$ and $\varphi$ as a retract in a way that is compatible with the mappings into $X$ and $Y'$. There are many examples  of quasi pullbacks which are not pullbacks. For instance, for any (non-unitary) space $X$, the following diagram 
\begin{eqnarray*}
\xymatrix{ \rule{3mm}{0mm}& X \ar[r]^{} \ar[d]_{} & \ast \ar[d]^{} & \\ &
       \ast  \ar[r]_{} &  \ast &}
\end{eqnarray*} is a quasi pullback however it is not a pullback. 

\medskip Next, we recall the notion of L-S category which, in our setting, is one bigger than the one given in \cite[Definition 1.1, p. 1]{cornea2003lusternik}. 

\begin{definition}
The \textit{Lusternik-Schnirelmann category} (L-S category) or category of a space $X$, denoted by cat$(X)$, is the least integer $m$ such that $X$ can be covered by $m$ open sets, all of which are contractible within $X$. We set $\text{cat}(X)=\infty$ if no such $m$ exists.
\end{definition}

We have $\text{cat}(X)=1$ iff $X$ is contractible. The L-S category is a homotopy invariant, i.e., if $X$ is homotopy equivalent to $Y$ (which we shall denote by $X\simeq Y$), then $\text{cat}(X)=\text{cat}(Y)$. 

\medskip Let $\dim X$ be the \textit{covering dimension} of a paracompact space $X$, i.e., $\dim X\leq n$ if and only if any open cover of $X$ has a locally finite open refinement such that no point of $X$ belongs to more than $n+1$ open sets of the refinement. The covering dimension of a manifold is the same as the dimension of the manifold. A space $X$ is called \textit{locally contractible} if any point of $X$ has an open neighborhood $U\subset X$ such that the inclusion $U\hookrightarrow X$ is null-homotopic \cite[p. 215]{farber2003}. 

\medskip The invariant $\mathrm{cat}(-)$ satisfies the following properties.

\begin{lemma}\label{cat-stimates}
\noindent
\begin{enumerate}
    \item \cite[Proposition 5.1, p. 336]{james1978}, \cite[p. 215]{farber2003} If $X$ is a path-connected paracompact locally contractible space, then \[\mathrm{cat}(X)\leq \mathrm{hdim}(X)+1,\]  where $\mathrm{hdim}(X)$ denotes the homotopical dimension of $X$, i.e., the minimal covering dimension of paracompact spaces having the homotopy type of $X$.
    
    \item \cite[Proposiç\~{a}o 4.1.34, p. 108]{zapata2022} We have $$\mathrm{cat}(X)\geq \mathrm{nil}\left(\widetilde{h}^\ast(X)\right),$$ where $\widetilde{h}^\ast(X)$ is any multiplicative reduced cohomology theory.
\end{enumerate}
\end{lemma}

\medskip Furthermore, we recall the following statements. 

 \begin{lemma}\label{prop-secat-map}
 Let $p:E\to B$ be a fibration.
 \begin{enumerate}
 \item \cite[Lemma 3.5]{zapata2023} If the following square
\begin{eqnarray*}
\xymatrix{ E^\prime \ar[r]^{\,\,} \ar[d]_{p^\prime} & E \ar[d]^{p} & \\
       B^\prime  \ar[r]_{} &  B &}
\end{eqnarray*}
is a quasi pullback, then $\mathrm{secat}\hspace{.1mm}(p^\prime)\leq \mathrm{secat}\hspace{.1mm}(p)$. 
\item \cite[Proposition 3.6]{zapata2023} For any space $Z$, we have \[\mathrm{secat}(p\times 1_Z)=\mathrm{secat}(p).\]  
     \item \cite[Theorem 18, p. 108]{schwarz1966} We have $\mathrm{secat}\hspace{.1mm}(p)\leq \mathrm{cat}(B)$. 
 \end{enumerate}
\end{lemma}

We recall that we work with sectional category for fibrations, the following comment is relevant for our study.

\begin{remark}\label{trivial-crosssection} Given a Hausdorff space $X$ admitting a fixed-point free involution $\tau$, we have that:
\begin{itemize}
    \item[i)] The quotient map $q_{(X,\tau)}:=q\colon X\to X/\tau$ is a fibration \cite[Theorem 3.2.2, pg. 66]{tom2008}. Therefore, it is possible to speak of its sectional category.
    \item[ii)] The quotient map $q:X\to X/\tau$ is a fiber bundle \cite{aguilar2002}, and it is the trivial bundle if and only if it admits a global cross-section \cite[p. 36]{steenrod1951}. In particular, in the case that $X$ is path-connected, the quotient map $q:X\to X/\tau$ is not the trivial bundle and does not admit a global cross-section. 
\end{itemize}
\end{remark}

Note that, the existence of a free action of $\mathbb{Z}_2$ on $X$ is equivalent to the existence of a fixed-point free involution $\tau:X\to X$.

\begin{example}\label{projective}
\noindent
\begin{enumerate}
    \item \cite[Example 3.11]{zapata2023} For any $m\geq 1$, we have \[\mathrm{secat}(q:S^m\to \mathbb{R}P^m)=\mathrm{cat}(\mathbb{R}P^m)=\mathrm{nil}\left(\text{Ker}(q^\ast_{\mathbb{Z}_2})\right)=m+1,\] where $q^\ast_{\mathbb{Z}_2}:H^\ast\left(\mathbb{R}P^m;\mathbb{Z}_2\right)\to H^\ast\left(S^m;\mathbb{Z}_2\right)$ is the induced homomorphism in $\mathbb{Z}_2$-cohomology.
    \item \cite[Lemma 3.12]{zapata2023}, \cite{roth2008} We have that $\mathrm{secat}\left(q^{\mathbb{R}^n}:F(\mathbb{R}^n,2)\to D(\mathbb{R}^n,2)\right)=n$ for any $n\geq 1$. 
    \item \cite[Lemma 3.13]{zapata2023} We have that $\mathrm{secat}\left(q^{S^n}:F(S^n, 2)\to D(S^n,2)\right)=n+1$ for any $n\geq 1$.
\end{enumerate}
\end{example}

From \cite{tohl1976}, we say that an involution $\tau$ on $X$ is \textit{equivalent} to an involution $\tau'$ on $X'$, denoted by $(X,\tau)\equiv(X',\tau')$, if there exists a homeomorphism $h:X\to X'$ such that $h\circ \tau=\tau'\circ h$. For our purpose, we present the following definition.  

\begin{definition}\label{comparable}
 We say that that an involution $\tau$ on $X$ is \textit{comparable} to an involution $\tau'$ on $X'$, denoted by $(X,\tau)\sim(X',\tau')$, if and only if there are $\mathbb{Z}_2$-equivariant maps $\varphi:(X,\tau)\to (X',\tau')$ and $\psi:(X',\tau')\to (X,\tau$). 
\end{definition}

The relation \aspas{$\equiv$} is an equivalence relation. In addition, $\sim$ is also an equivalence relation. Note that, if $(X,\tau)\equiv(X',\tau')$ then $(X,\tau)\sim(X',\tau')$. The converse does not hold. For instance, the antipodal involution on the $m$-dimensional sphere $S^m$ is comparable with the natural involution, $(x,y)\mapsto (y,x)$, on $F(S^m,2)$; and it is not equivalent. The given example above  has the property that the two spaces have the same homotopy type. So one  can  ask:  Does  there exist examples  where they are comparable but do not have the same homotopy type? The answer  is yes. Take  the spaces $X\times  F$ and  $F$. If  $\tau$ is a free involution  on  $F$, consider the free involution $1_X\times \tau$  on    $X\times  F$. The projection  $X\times F  \to  F$ and  the injective  map  $F\to  X\times  F$ given  by  $x\mapsto (x_0, x)$ for  an  arbitrary 
 point  $x_0\in  X$,   shows that they are comparable. But  they  do not  have the  same homotopy  type  for many  spaces  $F$.  This  in particular shows  that     if 
we define the relation \aspas{equivalent} replacing homeomorphism  by homotopy equivalence, the two concepts \aspas{comparable}  and  \aspas{equivalent}  are  not the  same.  

\medskip We will use the following statement. It follows from \cite[p. 61]{schwarz1966}.

\begin{remark}\label{pullback} Suppose that $X$ and $X'$ are Hausdorff spaces and $\tau:X\to X$, $\tau':X'\to X'$ are fixed-point free involutions. Then, any commutative diagram of the form  \begin{eqnarray*}
\xymatrix@C=2cm{ 
       X  \ar[r]^{\varphi}\ar[d]_{q} &  X'\ar[d]^{q'} &\\
      X/\tau \ar[r]_{\overline{\varphi} } & X'/\tau' &
       }
\end{eqnarray*} where $\varphi:X\to X'$ is a $\mathbb{Z}_2$-equivariant map and $\overline{\varphi}$ is induced by $\varphi$ in the quotient spaces, is a pullback since $\varphi$ restricts to a homeomorphism on each fiber (in this case both $q:X\to X/\tau$ and $q':X'\to X'/\tau'$ are $2$-sheeted covering maps).
\end{remark}

Then, we have:

\begin{lemma}\label{comparable-invariant}
    Suppose that $X, X'$ are Hausdorff spaces and $\tau\colon X\to X,\tau'\colon X'\to X'$ are fixed-point free involutions. If $(X,\tau)\sim(X',\tau')$, then \[\mathrm{secat}\left(q_{(X,\tau)}\colon X\to X/\tau\right)=\mathrm{secat}\left(q_{(X',\tau')}\colon X'\to X'/\tau'\right).\]
\end{lemma}
\begin{proof}
 By Definition~\ref{comparable}, there are $\mathbb{Z}_2$-equivariant maps $\varphi:(X,\tau)\to (X',\tau')$ and $\psi:(X',\tau')\to (X,\tau')$. Then, by Remark~\ref{pullback}, we have the following pullbacks:
 \begin{eqnarray*}
\xymatrix@C=2cm{ 
       X  \ar[r]^{\varphi}\ar[d]_{q_{(X,\tau)}} &  X'\ar[d]^{q_{(X',\tau')}} &\\
      X/\tau \ar[r]_{\overline{\varphi} } & X'/\tau' &
       } & \xymatrix@C=2cm{ 
       X'  \ar[r]^{\psi}\ar[d]_{q_{(X',\tau')}} &  X\ar[d]^{q_{(X,\tau)}} &\\
      X'/\tau' \ar[r]_{\overline{\psi} } & X/\tau &
       }
\end{eqnarray*} Then, by Item 1) from Lemma~\ref{prop-secat-map}, we conclude that $\mathrm{secat}\left(q_{(X,\tau)}\right)=\mathrm{secat}\left(q_{(X',\tau')}\right).$
\end{proof}

From Lemma~\ref{comparable-invariant} follows that if $(X,\tau)\equiv(X',\tau')$ then $\mathrm{secat}\left(q_{(X,\tau)}\right)=\mathrm{secat}\left(q_{(X',\tau')}\right)$.

\section{Parametrized Borsuk-Ulam theory via sectional category}\label{bu}
In this Section we present the notion of parametrized Borsuk-Ulam property (Definition~\ref{defn:pbup}) together with basic results about this property via sectional category. Our main result is Theorem \ref{theorem-principal},  which is proved at the end of the section. 

\medskip Let $\left((p,\tau);p'\right)$ be a triple where $p:E\to B$, $p':E'\to B$ are fibrations over $B$ and $\tau:E\to E$ is a free involution with $p\circ\tau=p$.
\begin{eqnarray*}
\xymatrix{ 
       E  \ar[rr]^{\tau}\ar[rd]_{p} & & E\ar[ld]^{p} \\
      &B & 
       }
\end{eqnarray*}

Motivated by \cite{dold1988}, we present the following definition.

\begin{definition}\label{defn:pbup}
 We say that $\left((p,\tau);p'\right)$ satisfies the \textit{parametrized Borsuk-Ulam property} (which we shall routinely abbreviate to PBUP) if  for every map $f:E\to E'$ with $p'\circ f=p$
\begin{eqnarray*}
\xymatrix{ 
       E  \ar[rr]^{f}\ar[rd]_{p} & & E'\ar[ld]^{p'} \\ 
      &B & 
       }
\end{eqnarray*}
there exists a point $x\in E$ such that $f(\tau(x))=f(x)$.   
\end{definition}

\begin{remark}
Note that, if there is a map $f:E\to E'$ with $p'\circ f=p$ then \[\mathrm{secat}(p)\geq \mathrm{secat}(p').\] For this reason, to study the PBUP we will suppose that the inequality $\mathrm{secat}(p)\geq \mathrm{secat}(p')$ always holds. Otherwise, the PBUP does not make sense. In addition, note that if $\mathrm{secat}(p')=1$ then there is a map $f:E\to E'$ with $p'\circ f=p$.
\end{remark}

We will fix some notation that will be used throughout the paper. From~\cite[Definition 2.7, p. 187]{cohen2010}, the \textit{$2$-ordered fibre-wise configuration space} of $p':E'\to B$ is the topological subspace \[F_{p'}(E',2)=\{(y_1,y_2)\in F(E',2):~p'(y_1)=p'(y_2)\}\] where $F(E',2)=\{(y_1,y_2)\in E'\times E':~y_1\neq y_2\}$ is the ordered configuration space of $2$ distinct points on $E'$ topologised as a subspace of the Cartesian power $E'\times E'$ \cite{fadell1962configuration}. We have the diagram

\[
\xymatrix@C=3cm{ F_{p'}(E',2) \ar@{^{(}->}[r]^{} \ar@{^{(}->}[d]_{} & F(E',2)\ar@{^{(}->}[d]_{} & \\
       E'\times_B E'  \ar@{^{(}->}[r]_{ } &  E'\times E' &}
\] where $E'\times_B E'=\{(y_1,y_2)\in E'\times E':~p'(y_1)=p'(y_2)\}$ is the fibered product of $p'$ and note that $F_{p'}(E',2)=\left(E'\times_B E'\right)\cap  F(E',2)$. In addition, $E'\times_B E'$ fiber over $B$ with fiber $F'\times F'$. Furthermore, $F_{p'}(E',2)$ fiber over $B$ with fiber $F(F',2)$ whenever $p':E'\to B$ is a locally trivial bundle and $B$ is paracompact (see \cite[Lemma 11]{goncalves2023}):
\[
\xymatrix@C=2cm{ F(F',2) \ar@{^{(}->}[rr]^{} \ar@{^{(}->}[d]_{} & & F'\times F'\ar@{^{(}->}[d]_{} & \\
       F_{p'}(E',2)  \ar@{^{(}->}[rr]_{ } \ar[dr]_{p} & & E'\times_B E'\ar[dl]^{p} &\\
        & B & &}
\]

Consider the double cover $q^{p'}:F_{p'}(E',2)\to F_{p'}(E',2)/\mathbb{Z}_2$ from the ordered fiber-wise configuration space $F_{p'}(E',2)$ to its unordered quotient $F_{p'}(E',2)/\mathbb{Z}_2$ given by the obvious free action $\tau_2$ of the symmetric group $\Sigma_2$ on $2$ letters, i.e., $\tau_2(y_1,y_2)=(y_2,y_1)$. 

\medskip From~\cite{zapata2023} for the triple $\left((X,\tau);Y\right)$, where $X,Y$ are Hausdorff spaces and $\tau:X\to X$ is a fixed-point free involution,  we say that $\left((X,\tau);Y\right)$ satisfies the \textit{Borsuk-Ulam property} (which we shall routinely abbreviate to BUP) if  for every map $f:X\to Y$ there exists a point $x\in X$ such that $f(\tau(x))=f(x)$. 

\medskip For Hausdorff spaces $X,Y$ and a fixed-point free involution $\tau:X\to X$, in \cite{zapata2023}, the authors discover a connection between the sectional category of the double covers $q:X\to X/\tau$ and $q^Y:F(Y,2)\to D(Y,2)$ from the ordered configuration space $F(Y,2)$ to its unordered quotient $D(Y,2)=F(Y,2)/\Sigma_2$, and the Borsuk-Ulam property (BUP) for the triple $\left((X,\tau);Y\right)$. The paracompactness hypothesis in \cite[Theorem 3.14]{zapata2023} was used to have that the quotient map $ X\to X/\tau$ is a fibration and hence it is possible to speak of its sectional category and we can use Item (1) from Lemma~\ref{prop-secat-map}. From Item (i) of Remark~\ref{trivial-crosssection}, the paracompactness hypothesis in \cite[Theorem 3.14]{zapata2023} is not necessary. So we can state the following statement without the hypotheses of paracompactnes.

\begin{lemma}\cite[Theorem 3.14]{zapata2023}\label{bup-secat}
Suppose that $X$ and $Y$ are Hausdorff spaces and $\tau:X\to X$ is a fixed-point free involution. If the triple $\left((X,\tau);Y\right)$ does not satisfy the BUP then \[\mathrm{secat}(q)\leq \mathrm{secat}(q^Y).\] Equivalently, if $\mathrm{secat}(q)> \mathrm{secat}(q^Y)$ then the triple $((X,\tau);Y)$ satisfies the BUP.
\end{lemma}

On the other hand, we present the following natural remarks.

\begin{remark}\label{basic-remark}
Let $p:E\to B$, $p':E'\to B$ be fibrations over $B$ and $\tau:E\to E$ is a fixed-point free involution with $p\circ\tau=p$.
\begin{enumerate}
    \item[1)] Fix $b\in B$, and consider the fiber $F_b=p^{-1}(b)$ over $b$. The commutativity of the triangle   \begin{eqnarray*}
\xymatrix{ 
       E  \ar[rr]^{\tau}\ar[rd]_{p} & & E\ar[ld]^{p} \\
      &B & 
       }
\end{eqnarray*} implies $\tau(F_b)=F_b$ and thus $\tau$ restricts to a fixed-point free involution on $F_b$.

\item[2)] Let $F$ be the fiber of $p$ and $F'$ be the fiber of $p'$ over a base point $b_0\in B$. Any map $f:E\to E'$ with $p'\circ f=p$ satisfy $f(F)\subset F'$. Thus, $f$ restricts to $f_{\mid}:F\to F'$.

\item[3)] Item 1) and 2) show that if the triple $\left((F,\tau);F'\right)$ satisfies the BUP then the triple $\left((p,\tau);p'\right)$ satisfies the PBUP. Equivalently, if the triple $\left((p,\tau);p'\right)$ does not satisfy the PBUP then $\left((F,\tau);F'\right)$ does not satisfy the BUP (and thus, by Lemma~\ref{bup-secat}), $\mathrm{secat}(q:F\to F/\tau)\leq\mathrm{secat}(q^{F'}:F(F',2)\to D(F',2))$. Therefore, if $\mathrm{secat}(q:F\to F/\tau)>\mathrm{secat}(q^{F'}:F(F',2)\to D(F',2))$ then the triple $\left((F,\tau);F'\right)$ satisfies the BUP and $\left((p,\tau);p'\right)$ satisfies the PBUP.

 \begin{eqnarray*}
\xymatrix{ 
F\ar@{^{(}->}[d]_{ }\ar[r]^{\tau}&F\ar@{^{(}->}[d]_{ } & & F'\ar@{^{(}->}[d]_{ }\\
   E\ar[r]^{\tau}\ar[rrdd]_{p}  &  E  \ar[rdd]_{p} & & E'\ar[ldd]^{p'} \\
   & & & \\
    &  &B & 
       }
\end{eqnarray*}
\end{enumerate}
\end{remark}   

From~Example~\ref{projective}, recall that $\mathrm{secat}(q:S^n\to \mathbb{R}P^n)=n+1$ and $\mathrm{secat}(q^{\mathbb{R}^n}:F(\mathbb{R}^n,2)\to D(\mathbb{R}^n,2))=n$ for any $n\geq 1$. Then we have the following statement.

\begin{proposition}\label{borsuk-and-parametrized}
Let $p:E\to B$ be a fibration with fiber $S^{m-1}$, $p':E'\to B$ be a fibration with fiber $\mathbb{R}^n$ and $\tau:E\to E$ is a fixed-point free involution with $p\circ\tau=p$ and the restriction $\tau:S^{m-1}\to S^{m-1}$ is comparable with the antipodal involution. If $n<m$ then the triple $\left((S^{m-1},\tau);\mathbb{R}^n\right)$ satisfies the BUP and also $\left((p,\tau);p'\right)$ satisfies the PBUP.
\end{proposition} 
\begin{proof}
It follows from Lemma~\ref{bup-secat} and Remark~\ref{basic-remark} together with Lemma~\ref{comparable-invariant}. 
\end{proof}

Proposition~\ref{borsuk-and-parametrized} also follows from the Famous Dold's theorem \cite[Corollary 1.5]{dold1988} when $p$ is the sphere-bundle of a vector bundle with fibre-dimension $m$ and $p'$ is a vector bundle of fibre-dimension $n$ over a paracompact space $B$ and $m>n$. The Dold's theorem is stronger, in fact Dold estimates the dimension of the set $Z_f=\{e\in E:~f(e)=f(\tau(e))\}$.

\medskip From Remark~\ref{pullback} we have the following proposition.

\begin{proposition}\label{inequality-secat-pullback}
Let $F\hookrightarrow E\stackrel{p}{\to} B$, $F'\hookrightarrow E'\stackrel{p'}{\to} B$ be fibrations with Hausdorff total spaces, and $\tau:E\to E$ be a fixed-point free involution with $p\circ\tau=p$.
\begin{enumerate}
    \item[1)]  The commutative  diagram \begin{eqnarray*}
\xymatrix@C=2cm{ 
       F  \ar@{^{(}->}[r]^{i}\ar[d]_{q_F} &  E\ar[d]^{q_E} &\\
      F/\tau \ar[r]_{\overline{i} } & E/\tau&
       }
\end{eqnarray*} where $i:F\hookrightarrow E$ is the inclusion map and $\overline{i}$ is induced by $i$ in the quotient spaces, is a pullback and thus the inequality $\mathrm{secat}(q_F)\leq\mathrm{secat}(q_E)$ holds. 
    
    \item[2)] The commutative  diagram \begin{eqnarray*}
\xymatrix@C=2cm{ 
       F_{p'}(E',2)  \ar@{^{(}->}[r]^{i}\ar[d]_{q^{p'}} &  F(E',2)\ar[d]^{q^{E'}} &\\
      F_{p'}(E',2)/\mathbb{Z}_2 \ar[r]_{\overline{i} } & D(E',2)&
       }
\end{eqnarray*} where $i:F_{p'}(E',2)\hookrightarrow F(E',2)$ is the inclusion map and $\overline{i}$ is induced by $i$ in the quotient spaces, is a pullback and thus the nequality   $\mathrm{secat}(q^{p'})\leq\mathrm{secat}(q^{E'})$ holds. In addition, we have the following pullback:
\begin{eqnarray*}
\xymatrix@C=2cm{ 
       F(F',2)  \ar@{^{(}->}[r]^{i}\ar[d]_{q^{F'}} &  F_{p'}(E',2)\ar[d]^{q^{p'}} &\\
      D(F',2) \ar[r]_{\overline{i} } & F_{p'}(E',2)/\mathbb{Z}_2&
       }
\end{eqnarray*} and thus, $\mathrm{secat}(q^{F'})\leq \mathrm{secat}(q^{p'})$.

\item[3)] Given a map $f:E\to E'$ satisfying $p'\circ f=p$, consider $Z_f=\{x\in E:~f(\tau(x))=f(x)\}$. Suppose that $Z_f\neq\varnothing$ and note that $\tau$ restricts to an involution on $Z_f$. Moreover, we have the following pulbback \begin{eqnarray*}
\xymatrix@C=2cm{ 
       Z_f  \ar@{^{(}->}[r]^{i}\ar[d]_{q_{Z_f}} &  E\ar[d]^{q_E} &\\
      Z_f/\tau \ar[r]_{\overline{i} } & E/\tau&
       }
\end{eqnarray*} Then, $\mathrm{secat}\left(q_{Z_f}\right)\leq\mathrm{secat}\left(q_{E}\right)$. In addition, from \cite[Proposition 58, p. 133]{schwarz1966}, we obtain that $\mathrm{secat}\left(q_{Z_f}\right)+\mathrm{secat}(q^{p'})\geq \mathrm{secat}\left(q_{E}\right)$.
\end{enumerate}
\end{proposition}

In the following example we show that, in item $1)$ from Proposition~\ref{inequality-secat-pullback}, the strict inequality or equality can be hold.  

\begin{example}\label{estrict}
\noindent
\begin{enumerate}
    \item[1)] Set $E=B\times F$ and $p=p_1:B\times F\to B$ the first coordinate projection and $\tau:F\to F$ a fixed-point free involution. We consider the fixed-point free involution $1_B\times\tau:B\times F\to B\times F,~(b,e)\mapsto (b,\tau(e))$. Then $(B\times F)/(1_B\times\tau)=B\times (F/\tau)$ and $q_E=1_B\times q_F$, where $q_F:F\to F/\tau$ is the quotient map. Thus $\mathrm{secat}(q_E)=\mathrm{secat}(q_F)$ (see Item (2) from Lemma~\ref{prop-secat-map}). 
    
    \item[2)] Let $h:S^3\to S^2$ be the Hopf fibration whose fiber is $S^1$, i.e., $h(a,b,c,d)=\left(a^2+b^2-c^2-d^2,2(ad+bc),2(bd-ac)\right)$. Note that, the antipodal map $A:S^3\to S^3,~A(a,b,c,d)=(-a,-b,-c,-d)$ satisfies $h\circ A=h$. Furthermore, the sectional category $\mathrm{secat}\left(q_{S^1}:S^1\to \mathbb{R}P^1\right)=2$ and $\mathrm{secat}\left(q_{S^3}:S^3\to \mathbb{R}P^3\right)=4$ (see Item (1) from Example~\ref{projective}).
\end{enumerate}
\end{example}

We observe that the proof of Dold's theorem \cite[Theorem 1.3]{dold1988} implies a lower bound for $\mathrm{secat}\left(q_{Z_f}\right)$.

\begin{proposition}\label{prop:lower-bound-secat-zf}
 Let $p:E\to B$, $p:E'\to B$ be vector bundles of fiber dimensions $m,n$ over a paracompact space $B$. Suppose that $E$ and $E'$ are Hausdorff spaces and $m-n\geq 2$. Let $S^{m-1}\hookrightarrow SE\stackrel{p_|}{\to} B$ be the sphere bundle of $p$ and $\tau:SE\to SE,x\mapsto \tau(x)=-x$ be the antipodal action. Given a map $f:SE\to E'$ satisfying $p'\circ f=p_|$, consider $Z_f=\{x\in SE:~f(\tau(x))=f(x)\}$ and suppose that $Z_f$ is path-connected. Then \[\mathrm{secat}\left(q_{Z_f}\colon Z_f\to Z_f/\tau\right)\geq \mathrm{nil}\left(\widetilde{\check{H}^\ast}(B;\mathbb{Z}_2)\right)+m-n-1,\] where $\widetilde{\check{H}^\ast}(-;\mathbb{Z}_2)$ denotes the reduced \v{C}ech cohomology.  
\end{proposition}
\begin{proof}
    Let $\overline{u}\in \check{H}^1(Z_f/\tau;\mathbb{Z}_2)$ be the characteristic class of the 2-sheeted covering map $q_{Z_f}:Z_f\to Z_f/\tau$. Note that $\overline{u}\neq 0$ and  $\overline{u}\in\mathrm{Ker}\left({q_{Z_f}}^\ast\right)$, where ${q_{Z_f}}^\ast:\check{H}^\ast(Z_f/\tau;\mathbb{Z}_2)\to \check{H}^\ast(Z_f;\mathbb{Z}_2)$ is the induced homomorphism in cohomology. Suppose that $\mathrm{nil}\left(\widetilde{\check{H}^\ast}(B;\mathbb{Z}_2)\right)=\ell+1$ and consider $\alpha_1,\ldots,\alpha_\ell\in \widetilde{\check{H}^\ast}(B;\mathbb{Z}_2)$ such that $\alpha:=\alpha_1\cup\cdots\cup\alpha_\ell\neq 0$. By \cite[Corollary 1.5]{dold1988} we obtain that $\alpha\cup \overline{u}^{m-n-1}\neq 0$ in $\widetilde{\check{H}^\ast}(Z_f/\tau;\mathbb{Z}_2)$. In addition, note that $\alpha\cup \overline{u}^{m-n-1}\in \mathrm{Ker}\left({q_{Z_f}}^\ast\right)$ (here, we use that $m-n\geq 2$). Then, by Lemma~\ref{prop-sectional-category}, $\mathrm{secat}\left(q_{Z_f}\right)\geq \ell+m-n-1+1=\mathrm{nil}\left(\widetilde{\check{H}^\ast}(B;\mathbb{Z}_2)\right)+m-n-1$.  
\end{proof}

In addition, it is easy to check the following topological criterion for the parametrized BUP (cf. \cite[Lemma 10]{goncalves2023}). 

\begin{proposition}\label{top-pbup} 
The triple $\left((p,\tau);p'\right)$ does not satisfy the PBUP if and only if there exists a $\mathbb{Z}_2$-equivariant map $\varphi:E\to F_{p'}(E',2)$ such that the following diagram commutes
\begin{eqnarray*}
\xymatrix@C=2cm{ 
       E \ar[rd]^{p}  \ar@{-->}[rr]^{\varphi}\ar[dd]_{q_E} & & F_{p'}(E',2)\ar[dd]^{q^{p'}} \ar[ld]^{p'} &\\
       & B & &\\
      E/\tau \ar@{-->}[rr]_{\overline{\varphi}} \ar[ru]^{\overline{p}} & & F_{p'}(E',2)/\mathbb{Z}_2 \ar[lu]^{\overline{p'}} &
       }
\end{eqnarray*} where $q_E:E\to E/\tau$ and $q^{p'}:F_{p'}(E',2)\to F_{p'}(E',2)/\mathbb{Z}_2$ are the $2$-sheeted covering maps, and $\overline{\varphi}$ is induced by $\varphi$ in the quotient spaces. 
\end{proposition}

Our main result is as follows. It gives conditions in terms of sectional category to have the PBUP. Let $p:E\to B$, $p':E'\to B$ be fibrations over $B$ and $\tau:E\to E$ is a fixed-point free involution with $p\circ\tau=p$.

\begin{theorem}[Principal theorem]\label{theorem-principal} Suppose that $E$ and $E'$ are Hausdorff spaces. If the triple $\left((p,\tau);p'\right)$ does not satisfy the PBUP or the triple $\left((E,\tau);F'\right)$ does not satisfy the BUP then \[\mathrm{secat}(q_E)\leq \mathrm{secat}(q^{p'}).\] Equivalently, if $\mathrm{secat}(q_E)>\mathrm{secat}(q^{p'})$ then the triple $\left((p,\tau);p'\right)$ satisfies the PBUP and the triple $\left((E,\tau);F'\right)$ satisfies the BUP.
\end{theorem}
\begin{proof}
It follows from Proposition~\ref{top-pbup} and Remark~\ref{pullback} together with Item (1) from Lemma~\ref{prop-secat-map} and Item 2) from Proposition~\ref{inequality-secat-pullback}.
\end{proof}

In the case that $p'=p$ we obtain the following statement.

\begin{remark}\label{rem:p-p}
  Let $F\hookrightarrow E\stackrel{p}{\to} B$ be a fibration and $\tau:E\to E$ is a fixed-point free involution with $p\circ\tau=p$, and suppose that $E$ is a Hausdorff space. Note that the triple $\left((p,\tau);p\right)$ does not satisfy the PBUP, then, by Theorem~\ref{theorem-principal}, \[\mathrm{secat}(q_E:E\to E/\tau)\leq\mathrm{secat}\left(q^p:F_p(E,2)\to F_p(E,2)/\mathbb{Z}_2\right).\] In addition, by Item 2) from Proposition~\ref{inequality-secat-pullback}, we have that \[\mathrm{secat}\left(q^F\right)\leq\mathrm{secat}\left(q^p\right)\leq\mathrm{secat}\left(q^E\right).\] Hence we have \[\max\{\mathrm{secat}(q_E),\mathrm{secat}\left(q^F\right)\}\leq\mathrm{secat}\left(q^p\right)\leq\mathrm{secat}\left(q^E\right).\]
\end{remark}

Also, note that Theorem~\ref{theorem-principal} implies the following statement. 

\begin{proposition}\label{prop:secat-quotient-conf}
 Suppose that $E$ and $E'$ are Hausdorff spaces. If $\mathrm{secat}(q_E)>\mathrm{secat}(q^{E'})$ then the triple $\left((p,\tau);p'\right)$ satisfies the PBUP and $\left((E,\tau);F'\right)$ satisfies the BUP.   
\end{proposition}
\begin{proof}
    Recall that the inequality $\mathrm{secat}(q^{E'})\geq \mathrm{secat}(q^{p'})$ holds, see Item 2) from Proposition~\ref{inequality-secat-pullback}. Thus, by Theorem~\ref{theorem-principal}, if $\mathrm{secat}(q_E)>\mathrm{secat}(q^{E'})$ then the triple $\left((p,\tau);p'\right)$ satisfies the PBUP and $\left((E,\tau);F'\right)$ satisfies the BUP\footnote{The proof of Proposition~\ref{prop:secat-quotient-conf} can be also obtain using Lemma~\ref{bup-secat}. Note that, if $\mathrm{secat}(q_E)>\mathrm{secat}(q^{E'})$ then the triple $\left((E,\tau);E'\right)$ satisfies the BUP. It is easy to check that, if $\left((E,\tau);E'\right)$ satisfies the BUP then $\left((p,\tau);p'\right)$ satisfies the PBUP and $\left((E,\tau);F'\right)$ satisfies the BUP.}.
\end{proof}

\section{Applications}\label{ap}
In this section we present applications of our results.
 
\medskip We start with the following remark.

\begin{remark}[To trivial fibration]\label{trivial-fibration}
Let $p:E\to B$ be a fibration with fiber $F$, $\tau:E\to E$ be a fixed-point free involution with $p\circ\tau=p$ and $F'$ be a space.
\begin{enumerate}
    \item[1)]  Note that, the triple $\left((E,\tau);F'\right)$ satisfies the BUP if and only if the triple $\left((p,\tau);p'\right)$ satisfies the PBUP, where $p':B\times F'\to B$ is the trivial fibration. Note that, any map $f:E\to B\times F'$ such that $p'\circ f=p$ can be written in the form $f=(p,g)$, where $g:E\to F'$.
    
    \begin{eqnarray*}
\xymatrix{ 
&F\ar@{^{(}->}[d]_{ } & & F'\ar@{^{(}->}[d]_{ }\\
   E\ar[r]^{\tau}\ar[rrdd]_{p}  &  E  \ar[rdd]_{p}\ar@{-->}[rr]^{f=(p,g)} & & B\times F'\ar[ldd]^{p'} \\
   & & & \\
    &  &B & 
       }
\end{eqnarray*}
    
    \item[2)] Set $E'=B\times F'$ and $p':E'\to B$ be the trivial fibration. Note that \[F_{p'}(E',2)=\{\left((b,f'),(b,f'')\right)\in F(B\times F',2):~b\in B\}\] and we can consider a $\mathbb{Z}_2$-equivariant homeomorphism  $\varphi:F_{p'}(E',2)\to B\times F(F',2)$, defined by $\varphi\left((b,f'),(b,f'')\right)=\left(b,(f',f'')\right)$ whose inverse $\psi:B\times F(F',2)\to F_{p'}(E',2)$ is given by $\psi\left(b,(f',f'')\right)=\left((b,f'),(b,f'')\right)$, where the involution in $B\times F(F',2)$ is the product $1_B\times \rho$ and $\rho:F(F',2)\to F(F',2),$ given by $\rho(f',f'')=(f'',f')$. Note that, the quotient map $B\times F(F',2)\to B\times F(F',2)/1_B\times \rho$ coincides with the product $1_B\times q^{F'}$. Thus, one has the following pullbacks:
    
    \begin{eqnarray*}
\xymatrix@C=1cm{ 
       F_{p'}(E',2) \ar[d]^{q^{p'}}  \ar[r]^{\varphi} & B\times F(F',2)\ar[d]^{1_B\times q^{F'}} \\
        F_{p'}(E',2)/\mathbb{Z}_2 \ar[r]_{\overline{\varphi}} & B\times D(F',2)\\
       } & & \xymatrix@C=1cm{ 
       B\times F(F',2)\ar[d]^{1_B\times q^{F'}}\ar[r]^{\psi} & F_{p'}(E',2) \ar[d]^{q^{p'}} \\
        B\times D(F',2) \ar[r]_{\overline{\psi}}& F_{p'}(E',2)/\mathbb{Z}_2  \\
       }
\end{eqnarray*} then $\mathrm{secat}(q^{p'})=\mathrm{secat}(1_B\times q^{F'})=\mathrm{secat}(q^{F'})$.
\end{enumerate}
\end{remark}

Remark~\ref{trivial-fibration} shows that Theorem~\ref{theorem-principal} is a generalization, in the parametrized setting, of \cite[Theorem 3.12]{zapata2022} (see Lemma~\ref{bup-secat}).

\medskip In the case that $F\hookrightarrow E\stackrel{p}{\to} B$ is a locally trivial bundle over a paracompact space $B$ we obtain an upper bound to $\mathrm{secat}(q^{p})$.

\begin{proposition}\label{prop:upper-bound-secat-fibered}
  Let $F\hookrightarrow E\stackrel{p}{\to} B$ be a locally trivial bundle and assume that the spaces $E$ and $B$ are metrizable. Suposse that $F$ is a topological manifold. Then \[\mathrm{secat}(q^{p})\leq 2\dim F+\dim B+1.\] 
\end{proposition}

\begin{proof}
    By Item 3) from Lemma~\ref{prop-secat-map} and Item 1) from Lemma~\ref{cat-stimates}, we obtain  \begin{align*}
     \mathrm{secat}(q^{p})&\leq \mathrm{cat}\left(E\times_BE_|/\mathbb{Z}_2\right)\\ 
     &\leq \dim \left(E\times_BE_|/\mathbb{Z}_2\right)+1.\\
    \end{align*}
On the other hand, note that $E\times_BE_|/\mathbb{Z}_2$ is the total space of a locally trivial bundle over $B$ with fiber $D(F,2)$ (see \cite[Lemma 11]{goncalves2023}) implying \begin{align*}
  \dim \left(E\times_BE_|/\mathbb{Z}_2\right)&\leq \dim D(F,2)+\dim B\\
  &=\dim F(F,2)+\dim B\\
  &= 2\dim F+\dim B.\\
\end{align*} Then, \begin{align*}
     \mathrm{secat}(q^{p})&\leq 2\dim F+\dim B+1.
\end{align*}    
\end{proof}

\begin{corollary}[To locally trivial bundle]
   Let $F'\hookrightarrow E'\stackrel{p'}{\to} B$ be a locally trivial bundle and assume that the spaces $E'$ and $B$ are metrizable. Suposse that $F'$ is a topological manifold. Let $F\hookrightarrow E\stackrel{p}{\to} B$ be a fibration, $\tau:E\to E$ be a fixed-point free involution with $p\circ\tau=p$ and suppose that $E$ is a Hausdorff space. If $\mathrm{secat}(q_E:E\to E/\tau)>2\dim F'+\dim B+1$ then the triple $\left((p,\tau);p'\right)$ satisfies the PBUP and $\left((E,\tau);F'\right)$ satisfies the BUP.
\end{corollary}
\begin{proof}
    It follows from Theorem~\ref{theorem-principal} together with Proposition~\ref{prop:upper-bound-secat-fibered}.
\end{proof}

\medskip Next, we will present direct applications of our results. The first example studies the PBUP between complex Hopf fibration and trivial fibration.

\begin{example}[Complex Hopf fibration-trivial fibration]\label{complexHopf-trivial}
Let $p:S^{2n+1}\to\mathbb{C}P^n$ be the complex Hopf fibration with fiber $S^1$ and $A:S^{2n+1}\to S^{2n+1},~A(z)=-z$, be the antipodal involution. Note that $p\circ A=p$ and $\mathrm{secat}\left(q_{S^{2n+1}}:S^{2n+1}\to \mathbb{R}P^{2n+1}\right)=2n+2$ (see Example~\ref{projective}). For any $m\geq 1$ we have that $\mathrm{secat}\left(q^{\mathbb{R}^m}:F(\mathbb{R}^m,2)\to D(\mathbb{R}^m,2)\right)=m$ (see Example~\ref{projective})). Then, by Theorem~\ref{theorem-principal} or Lemma~\ref{bup-secat} (together with Item 1) from Remark~\ref{trivial-fibration}), we have that the triple $\left((p,A);p'\right)$ satisfies the PBUP for any $m<2n+2$, where $p':\mathbb{C}P^n\times\mathbb{R}^m\to \mathbb{C}P^n$ is the trivial fibration. In contrast, note that the triple $\left((p,A);p'\right)$ does not satisfy the PBUP for any $m\geq 2n+2$, see Item 1) from Remark~\ref{trivial-fibration},

\begin{eqnarray*}
\xymatrix{ 
&S^1\ar@{^{(}->}[d]_{ } & & \mathbb{R}^m\ar@{^{(}->}[d]_{ }\\
   S^{2n+1}\ar[r]^{A}\ar[rrdd]_{p}  &  S^{2n+1}  \ar[rdd]_{p} & & \mathbb{C}P^n\times\mathbb{R}^m\ar[ldd]^{p'} \\
   & & & \\
    &  &\mathbb{C}P^n & 
       }
\end{eqnarray*}
\end{example}




Next, we explore the PBUP into the Fadell-Neuwirth fibration. For this purpose we persent the following remark. 

\begin{remark}[To Fadell-Neuwirth fibration]\label{Fadell-Neuwirth-fibration}
 Let $B$ be a connected topological manifold (without boundary) with dimension at least $2$, the Fadell-Neuwirth fibration $\pi=\pi^B:F(B,2)\to B$ is given by projection on the first coordinate, i.e., $\pi(b_1,b_2)=b_1$ for all $(b_1,b_2)\in F(B,2)$. From \cite{fadell1962configuration}, $\pi$ is a locally trivial bundle with fiber $B-\{b_0\}$, where $b_0$ is the base point of $B$. 
 
Let $p:E\to B$ be a fibration with fiber $F$ and $\tau:E\to E$ be a fixed-point free involution with $p\circ\tau=p$.
\begin{enumerate}
    \item[1)]  Note that, if the triple $\left((E,\tau);B\right)$ satisfies the BUP then the triple $\left((p,\tau);\pi\right)$ satisfies the PBUP, where $\pi:F(B,2)\to B$ is the Fadell-Neuwirth fibration. Note that, any map $\varphi:E\to F(B,2)$ such that $\pi\circ \varphi=p$ can be written in the form $\varphi=(p,g)$, where $g:E\to B$ and $p(x)\neq g(x)$ for all $x\in E$.
    
    \begin{eqnarray*}
\xymatrix{ 
&F\ar@{^{(}->}[d]_{ } & & B-\{b_0\}\ar@{^{(}->}[d]_{ }\\
   E\ar[r]^{\tau}\ar[rrdd]_{p}  &  E  \ar[rdd]_{p}\ar@{-->}[rr]^{\varphi=(p,g)} & & F(B,2)\ar[ldd]^{\pi} \\
   & & & \\
    &  &B & 
       }
\end{eqnarray*}
    
    \item[2)] Recall that $\pi:F(B,2)\to B$ is the Fadell-Neuwirth fibration. Note that \[F(B,2)\times_B F(B,2)_|=\{\left((b,b_2),(b,c_2)\right)\in F\left(F(B,2),2\right):~b\in B\}\] and we can consider a $\mathbb{Z}_2$-equivariant homeomorphism  $\varphi:F(B,2)\times_B F(B,2)_|\to F(B,3)$, defined by $\varphi\left((b,b_2),(b,c_2)\right)=\left(b,b_2,c_2\right)$ whose inverse $\psi:F(B,3)\to F(B,2)\times_B F(B,2)_|$ is given by $\psi\left(b,b_2,c_2\right)=\left((b,b_2),(b,c_2)\right)$, where the fixed-point free involution in $F(B,3)$ is the restriction of the product $1_B\times \rho$ and $\rho:F(B,2)\to F(B,2),$ given by $\rho(b_2,c_2)=(c_2,b_2)$. Thus, one has the following pullbacks:
    \begin{eqnarray*}
\xymatrix@C=2cm{ 
       F(B,2)\times_B F(B,2)_| \ar[d]^{q^{\pi}}  \ar[r]^{\varphi} & F(B,3)\ar[d]^{\eta} \\
        F(B,2)\times_B F(B,2)_|/\mathrm{Z}_2 \ar[r]_{\overline{\varphi}} & F(B,3)/1_B\times \rho\\
       } \\ \\ \xymatrix@C=2cm{ 
       F(B,3)\ar[d]^{\eta}\ar[r]^{\psi} & F(B,2)\times_B F(B,2)_| \ar[d]^{q^{\pi}} \\
        F(B,3)/1_B\times \rho \ar[r]_{\overline{\psi}}& F(B,2)\times_B F(B,2)_|/\mathrm{Z}_2  \\
       }
\end{eqnarray*} then $\mathrm{secat}(q^{\pi})=\mathrm{secat}(\eta)$. 

On the other hand, we have the following pullback:
\begin{eqnarray*}
\xymatrix@C=2cm{ 
       F(B,3)  \ar@{^{(}->}[r]^{i}\ar[d]_{\eta} &  B\times F(B,2)\ar[d]^{1_B\times q^B} &\\
      F(B,3)/1_B\times \rho \ar[r]_{\overline{i} } & B\times D(B,2)&
       }
\end{eqnarray*} and thus, $\mathrm{secat}(\eta)\leq \mathrm{secat}(q^B)$.
\end{enumerate}
\end{remark}

Remark~\ref{Fadell-Neuwirth-fibration}-$2)$ together with Theorem~\ref{theorem-principal} implies the following statement.

\begin{proposition}\label{to-fadell}
Let $p:E\to B$ be a fibration with fiber $F$, $\tau:E\to E$ be a fixed-point free involution with $p\circ\tau=p$ and $\pi:F(B,2)\to B$ be the Fadell-Neuwirth fibration. Assume that $E$ is a Hausdorff space. If $\mathrm{secat}(q_E:E\to E/\tau)>\mathrm{secat}(\eta:F(B,3)\to F(B,3)/1_B\times \rho)$ then the triple $\left((p,\tau);\pi\right)$ satisfies the PBUP and $\left((E,\tau);B-\{b_0\}\right)$ satisfies the BUP.
\end{proposition}
\begin{proof}
From Remark~\ref{Fadell-Neuwirth-fibration}-$2)$ we have $\mathrm{secat}(q^{\pi})=\mathrm{secat}(\eta)$. Then, $\mathrm{secat}(q_E:E\to E/\tau)>\mathrm{secat}(q^{\pi})$ and thus, by Theorem~\ref{theorem-principal}, one can conclude that the triple $\left((p,\tau);\pi\right)$ satisfies the PBUP and $\left((E,\tau);B-\{b_0\}\right)$ satisfies the BUP.
\end{proof}


For fibrations over the 3-sphere we have the following example. Note that, the Fadell-Neuwirth fibration $\pi:F(S^3,2)\to S^3$ has sectional category equal to one.

\begin{example}\label{s3}
Let $p:E\to S^3$ be a fibration with fiber $F$ and $\tau:E\to E$ be a fixed-point free involution with $p\circ\tau=p$ such that $\mathrm{secat}(q_E:E\to E/\tau)>3$, then the triple $\left((p,\tau);\pi\right)$ satisfies the PBUP, where  $\pi:F(S^3,2)\to S^3$ is the Fadell-Neuwirth fibration.

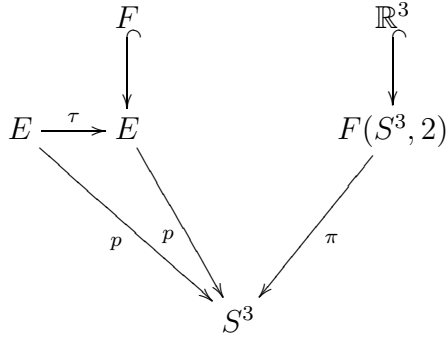
\begin{eqnarray*}
\xymatrix{ 
&F\ar@{^{(}->}[d]_{ } & & \mathbb{R}^3\ar@{^{(}->}[d]_{ }\\
   E\ar[r]^{\tau}\ar[rrdd]_{p}  &  E  \ar[rdd]_{p} & & F(S^3,2)\ar[ldd]^{\pi} \\
   & & & \\
    &  &S^3& 
       }
\end{eqnarray*} 

Indeed, we consider that $S^3$ is a topological group and thus we have a $\mathbb{Z}_2$-equivariant homeomorphism $\varphi:F(S^3,3)\to S^3\times F(S^3-\{1\},2),~(z_0,z_1,z_2)\mapsto \varphi(z_0,z_1,z_2)=\left(z_0,(z_1z_0^{-1},z_2z_0^{-1})\right)$ whose inverse $\psi:S^3\times F(S^3-\{1\},2)\to F(S^3,2)$ is given by $\psi\left(z_0,(z_1,z_2)\right)=(z_0,z_1z_0,z_2z_0)$, where we consider the involution over $S^3\times F(S^3-\{1\},2)$ by $1_{S^3}\times \rho$. Recall that, $\rho:F(S^3-\{1\},2)\to F(S^3-\{1\},2)$ is defined by $\rho(z_1,z_2)=(z_2,z_1)$. Then, we have the following pullbacks 

\begin{eqnarray*}
\xymatrix@C=1cm{ 
       F(S^3,3) \ar[d]^{\eta}  \ar[r]^{\varphi} & S^3\times F(\mathbb{R}^3,2)\ar[d]^{1_{S^3}\times q^{\mathbb{R}^3}} \\
        F(S^3,2)/1_{S^3}\times\rho \ar[r]_{\overline{\varphi}} & S^3\times D(\mathbb{R}^3,2)\\
       } & & \xymatrix@C=1cm{ 
       S^3\times F(\mathbb{R}^3,2)\ar[d]^{1_{S^3}\times q^{\mathbb{R}^3}}\ar[r]^{\psi} & F(S^3,3) \ar[d]^{\eta} \\
        S^3\times D(\mathbb{R}^3,2) \ar[r]_{\overline{\psi}}& F(S^3,2)/1_{S^3}\times\rho  \\
       }
\end{eqnarray*} then $\mathrm{secat}(\eta)=\mathrm{secat}(1_{S^3}\times q^{\mathbb{R}^3})=\mathrm{secat}(q^{\mathbb{R}^3})=3$. Since, $\mathrm{secat}(q_E:E\to E/\tau)>3=\mathrm{secat}(\eta)$ together with Proposition~\ref{to-fadell}, one can conclude that the triple $\left((p,\tau);\pi\right)$ satisfies the PBUP and $\left((E,\tau);\mathbb{R}^3\right)$ satisfies the BUP. We observe that $\mathrm{secat}\left(q^{S^3}\right)=4$.
\end{example}

\begin{remark}
The most appealing situation of Example~\ref{s3} holds for $\mathrm{secat}\left(q_E:E\to E/\tau\right)=4$, as in fact $\mathrm{secat}\left(q^{S^3}:F(S^3,2)\to D(S^3,2)\right)=4$, in view of the last part in Remark~\ref{Fadell-Neuwirth-fibration}-$2)$. Indeed, it would be interesting to know whether there is a fibration $E\to S^3$ with fiber $F$ for which $E$ admits a fixed-point free involution $\tau:E\to E$ with $p\circ\tau=p$ having $\mathrm{secat}\left(q_E:E\to E/\tau\right)=4$. 
\end{remark}

Furthermore, the following example presents the PBUP between Hopf fibration and the Fadell-Neuwirth fibration.

\begin{example}\label{s2}
From Example~\ref{estrict}-$2)$, recall that $h:S^3\to S^2$ is the Hopf fibration whose fiber is $S^1$, the antipodal map $A:S^3\to S^3,~A(a,b,c,d)=(-a,-b,-c,-d)$ satisfies $h\circ A=h$ and $\mathrm{secat}\left(q_{S^3}:S^3\to \mathbb{R}P^3\right)=4$. 

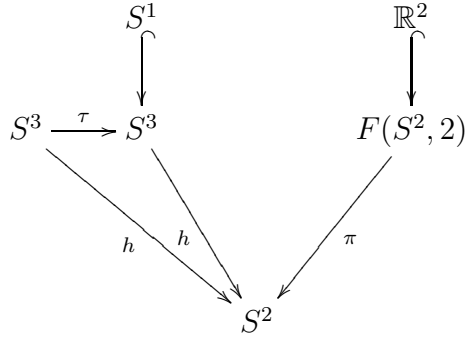
\begin{eqnarray*}
\xymatrix{ 
&S^1\ar@{^{(}->}[d]_{ } & & \mathbb{R}^2\ar@{^{(}->}[d]_{ }\\
   S^3\ar[r]^{\tau}\ar[rrdd]_{h}  &  S^3  \ar[rdd]_{h} & & F(S^2,2)\ar[ldd]^{\pi} \\
   & & & \\
    &  &S^2& 
       }
\end{eqnarray*} Note that, $\mathrm{secat}\left(q^{S^2}:F(S^2,2)\to D(S^2,2)\right)=3$, and by Remark~\ref{Fadell-Neuwirth-fibration}-$2)$, we have \[\mathrm{secat}\left(\eta:F(S^2,3)\to F(S^2,3)/1_{S^2}\times\rho\right)\leq\mathrm{secat}\left(q^{S^2}:F(S^2,2)\to D(S^2,2)\right)=3,\] then $\mathrm{secat}\left(q_{S^3}:S^3\to \mathbb{R}P^3\right)=4>3\geq\mathrm{secat}\left(\eta:F(S^2,3)\to F(S^2,3)/1_{S^2}\times\rho\right)$, and thus by Proposition~\ref{to-fadell}, the triple $\left((h,A);\pi\right)$ satisfy the PBUP. It also follows from Lemma~\ref{bup-secat} together with Remark~\ref{Fadell-Neuwirth-fibration}-$1)$ and the fact that  \[\mathrm{secat}\left(q_{S^3}:S^3\to \mathbb{R}P^3\right)=4>3=\mathrm{secat}\left(q^{S^2}:F(S^2,2)\to D(S^2,2)\right).\]
\end{example}






\bibliographystyle{plain}

\begin{thebibliography}{10}
\bibitem{aguilar2002} Aguilar, M. A., Gitler, S., Prieto, C., \& Aguilar, M. A. (2002). Algebraic topology from a homotopical viewpoint (Vol. 2002). New York: Springer.
\bibitem{tohl1976} Asoh, T. (1976). Classification of free involutions on surfaces. Hiroshima Mathematical Journal, 6(1), 171-181.
\bibitem{berstein1961} Berstein, I., \& Ganea, T. (1962). The category of a map and of a cohomology class. Fundamenta Mathematicae, 50(3), 265-279.
\bibitem{cohen2010} Cohen, F. R. (2010). Introduction to configuration spaces and their applications. In Braids (pp. 183-261). Hackensack, NJ: World Scientific Publisher.
\bibitem{cornea2003lusternik} Cornea, O., Lupton, G., Oprea, J., \& Tanré, D. (2003). Lusternik-Schnirelmann Category. Mathematical Surveys and Monographs, 103.
\bibitem{dold1988} Dold, A. (1988). Parametrized Borsuk-Ulam theorems. Commentarii Mathematici Helvetici, 63(1), 275-285.
\bibitem{fadell1962configuration} Fadell, E., \& Neuwirth, L. (1962). Configuration spaces. Mathematica Scandinavica, 10, 111-118.
\bibitem{farber2003} Farber, M. (2003). Topological complexity of motion planning. Discrete \& Computational Geometry, 29, 211-221.
\bibitem{goncalves2023} Gonçalves, D. L., Laass, V. C., \& Silva, W. L. (2023). The Borsuk-Ulam property for homotopy classes on bundles, parametrized braids groups and applications for surfaces bundles. arXiv preprint arXiv:2308.11445.
\bibitem{james1978} James, I. M. (1978). On category, in the sense of Lusternik-Schnirelmann. Topology, 17(4), 331-348.
\bibitem{roth2008} Roth, F. (2008). On the category of Euclidean configuration spaces and associated fibrations. Groups, homotopy and configuration spaces, 13, 447-461.
\bibitem{schwarz1958genus} Schwarz, A. S. (1958). The genus of a fiber space. Dokl. Akad. Nauk SSSR (NS). 119, 219--222.
\bibitem{schwarz1966} Schwarz, A. (1966). The genus of a fiber space. Amer. Math. Soc. Transl. Ser. 2 (55), 49--140.
\bibitem{steenrod1951} Steenrod, N. (1951). The topology of fibre bundles. IL, Miscow, 1953; Russian transl. of Princeton Math. Series, Vol. 14, Princeton Univ. Press, Princeton, New Jersey.
\bibitem{tom2008} tom Dieck, T. (2008). Algebraic topology. Vol. 8. European Mathematical Society.
\bibitem{zapata2022} Zapata, C.A.I. (2022). Espaços de configurações no problema de planificação de movimento simultâneo livre de colisões. Ph.D thesis, Universidade de São Paulo (in Portuguese).
\bibitem{zapata2023} Zapata, C.A.I., \& Gonçalves, D. L. (2023). Borsuk–Ulam Property and Sectional Category. Bulletin of the Iranian Mathematical Society, 49(4), 41.
\bibitem{zapata2023higher} Zapata, C.A.I., \& González, J. (2023). Higher topological complexity of a map. Turkish Journal of Mathematics, 47(6), 1616-1642.
\end{thebibliography}

\end{document}